\documentclass[12pt]{amsart}
\usepackage[utf8]{inputenc}
\usepackage{graphicx}

\usepackage{mathtools, bm}
\usepackage{amssymb, bm}
\usepackage{amsthm}
\usepackage{amsfonts}
\usepackage{stmaryrd}
\usepackage{hyperref, enumerate}
\usepackage{amsmath}
\usepackage[T1]{fontenc}
\usepackage{setspace}
\usepackage{dsfont}
\usepackage{array, color}
\usepackage{fancybox}
\usepackage{marvosym}
\usepackage{wasysym}
\usepackage{soul}
\usepackage{tikz-cd}
\usepackage{hhline, bm}

\newcommand{\Z}{\mathbb{Z}}

\newcommand{\Q}{\mathbb{Q}}

\newtheorem{thm}{Theorem}[section]
\newtheorem{lemma}[thm]{Lemma}

\newtheorem{prop}[thm]{Proposition}

\begin{document}
\title[]{Explicit uniform lower bounds for the canonical height on elliptic
curves over abelian extensions}
\author{Jonathan Jenvrin}
\address{Jonathan Jenvrin, Univ. Grenoble Alpes, CNRS, IF, 38000 Grenoble, France. \textit{E-mail adress :} \href{mailto:jonathan.jenvrin@univ-grenoble-alpes.fr}
{\texttt{jonathan.jenvrin@univ-grenoble-alpes.fr}}
}
\date{}
\maketitle

\begin{abstract}

We establish an explicit lower bound for the Néron–Tate height on elliptic curves with complex multiplication, for non-torsion points defined over the maximal abelian extension of a number field. Building on a strategy developed by Amoroso, David, and Zannier, we provide an alternative proof of a theorem originally due to Baker. The novelty in our approach is that it produces a lower bound that is fully explicit and independent of the discriminant of the base field.

\end{abstract}

\section{Introduction}
Let $F$ be a number field, $F^{\mathrm{ab}}$ its maximal abelian extension, and $\bar{F}$ a fixed algebraic closure of $F$.
Let $\mathcal{E}$ be an elliptic curve defined over $F$, and denote by 
\[
\hat{h} : \mathcal{E}(\bar{F}) \longrightarrow \mathbb{R}
\]
the Néron–Tate height on $\mathcal{E}(\bar{F})$ (see \cite[Definition on p.~248]{SilvermanArithmeticEC} for more details).  
We write $\mathcal{E}_{tors}$ for the torsion subgroup of $\mathcal{E}$, and for an integer $n \geq 1$, we denote by $\mathcal{E}[n]$ the subgroup of $n$-torsion points.  For an algebraic number $\alpha$, we denote by $h(\alpha)$ its absolute logarithmic Weil height (see \cite[Definition 1.5.4]{bombieri2006heights} for more details)

Building on the work of Amoroso, Dvornicich, and Zannier, in the multiplicative group $\mathbb{G}_{m}$, who established lower bounds for the Weil height of algebraic numbers lying in abelian extensions of number fields (\cite[Theorem on p.~261]{Amoroso-Dvornicich} and \cite[Theorem 1.1]{Amoroso-Zannier}), Baker (\cite[Theorem 1.1]{BakerMatthew}) obtained an analogous result for the Néron–Tate height of points on elliptic curves with complex multiplication (CM for short) defined over the maximal abelian extension of a number field.

\begin{thm}[{\cite[Theorem 1.1]{BakerMatthew}}] \label{Theorem BakerMatthew} Let $F$ be a number field, and let $\mathcal{E}$ be an elliptic curve with CM defined over $F$. Then there exists an effectively computable real constant $C>0$ (depending on $\mathcal{E}/F$) such that $\hat{h}(P) \geq C$ for all $P \in$ $\mathcal{E}\left(F^{\mathrm{ab}}\right) \setminus \mathcal{E}_{tors}$.
\end{thm}

The aim of this work is to provide an alternative proof of Theorem~\ref{Theorem BakerMatthew}, using the method developed in \cite[Proposition~4.2]{Amoroso_David_Zannier-OnFieldWithPropertyB}, which was not known at the time Baker wrote \cite{BakerMatthew}.  
In contrast with Baker’s original proof, which relies on the Chebotarev density theorem 
\cite[Section~4]{BakerMatthew} and produces a lower bound that depends also on the discriminant of $F$, 
the approach proposed here yields a lower bound with a constant $C$ that is entirely explicit, 
depending solely on the degree of the number field and the $j$-invariant of the elliptic curve. More precisely we prove the following:

\begin{thm} \label{goal}
   Let $F$ be a number field, and let $d=[F:\Q]$. Let $\mathcal{E}$ be an elliptic curve with CM defined over $F$, let $j_{\mathcal{E}}$ be the $j$-invariant of $\mathcal{E}$. Then, for all points $P \in \mathcal{E} \left(F^{ab}\right) \setminus \mathcal{E}_{\mathrm{tors}}$, we have
\[
\hat{h}(P) \geq  2^{-4(9+h(j_{\mathcal{E}}))d^{2}-32d-8}.
\]
\end{thm} In Section \ref{Section auxiliary result}, we introduce several auxiliary results.
The first, Lemma \ref{lemma minoration hauteur de neron}, is the analogue of \cite[Lemma~2.2]{Amoroso_David_Zannier-OnFieldWithPropertyB} in the setting of elliptic curves. Its proof follows the approach of Baker \cite[pp.~10–11]{BakerMatthew}. It provides a lower bound for the Néron-Tate height of a point, based on local height inequalities at places above a fixed prime. The proof relies on the decomposition of the global height into local heights, combined with lower bounds at non-archimedean places coming from good reduction and a uniform control of the archimedean contribution via known estimates.

The second result, Proposition~\ref{Theorem intermediare}, corresponds to \cite[Proposition~3.1]{Amoroso_David_Zannier-OnFieldWithPropertyB}, adapted to the context of elliptic curves. It establishes a uniform lower bound for the canonical height of a non-torsion point on an elliptic curve with CM, under a mild Galois-theoretic assumption. The proof distinguishes between tame and wild ramification of a fixed prime in an abelian extension, applying Lemma~\ref{lemma minoration hauteur de neron} in both cases to derive a local height inequality.

The third result, Proposition~\ref{question}, is a crucial input for constructing a non-torsion point with controlled bounded height in the later part of the proof of Theorem~\ref{goal}, in the same spirit as the construction of a non-root of unity in~\cite[Proposition~4.2]{Amoroso_David_Zannier-OnFieldWithPropertyB}.

The idea is to find two Galois automorphisms $\sigma_1, \sigma_2 \in \mathrm{Gal}(F(\mathcal{E}[q])/F)$ that act on $\mathcal{E}[q]$ as scalar multiplications by two integers $g_1$ and $g_2$, such that their difference is explicitly bounded by a constant depending only on the degree of the extension $F/\mathbb{Q}$. 

To this end, we exploit the structure of the Galois representation on torsion points, which lies in a commutative group since $\mathcal{E}$ has CM.

Finally, in Section \ref{proof of goal}, we prove Theorem \ref{goal}. The proof adapts the strategy of Amoroso, David, and Zannier~\cite[Proposition 4.2]{Amoroso_David_Zannier-OnFieldWithPropertyB} by constructing via Proposition \ref{question} a non-torsion linear combination of Galois conjugates of a given point. By applying Proposition~\ref{Theorem intermediare} to this linear combination, we derive a fully explicit inequality depending only on the degree of the number field and the $j$-invariant of the curve, without involving the discriminant of the base field.\\

Before proceeding to the proof of our main result, it is worth noting that Silverman extended Baker's work by proving \cite[Theorem 1]{Silverman} the non-CM analogue of \cite[Theorem 1.1]{BakerMatthew}. Our argument does not extend to the non-CM case, since a crucial ingredient  
(relied upon among others) is \cite[Theorem II.2.3]{AdvancedTopicEllipticCurve} which can be used only in the CM case.

Subsequently, Baker and Silverman generalize their respective results to abelian varieties of arbitrary dimension \cite[Theorem 0.1]{SilvermanBaker}. This generalization builds upon both of their earlier works and again relies on the Chebotarev density theorem \cite[Section 6]{SilvermanBaker}. 

It is natural to ask whether such a generalization, at least in the CM case, can be obtained without invoking the Chebotarev density theorem.

\section{Auxiliary results} \label{Section auxiliary result}
We start with the elliptic analogue of \cite[Lemma 2.2]{Amoroso_David_Zannier-OnFieldWithPropertyB}.
\begin{lemma} \label{lemma minoration hauteur de neron}
    Let $F$ be a number field and let \( \mathcal{E} \) be an elliptic curve over $F$ with everywhere potential good reduction.  
    Let $L/F$ be a finite Galois extension and let \( \sigma \in \mathrm{Gal}(L/F) \). 
    Let $p$ be a rational prime, $a,b\ge 1$ integers and $\rho>0$ a real number.
   Suppose that $P \in \mathcal{E}(L)$ satisfies the following two conditions: 
   \begin{itemize}
       \item \( aP - b\sigma(P) \neq 0 \)
       \item for every place $v$ of $L$ above $p$ \begin{equation} \label{eq:neron_local_lower_bound}
        \lambda_{v}(aP - b\sigma(P)) \geq \frac{\log p}{\rho}
    \end{equation}
    where \( \lambda_{v} \) is the Néron local height function associated to \( v \) (see \cite[Definition on p.456]{AdvancedTopicEllipticCurve}).
   \end{itemize}
    Then
    \begin{equation} \label{eq:canonical_height_lower_bound}
        \hat{h}(P) \geq \frac{1}{2(a^{2} + b^{2})} \left( \frac{\log p \cdot [F_{\wp} : \mathbb{Q}_{p}]}{\rho [F : \mathbb{Q}]} - C(j_{\mathcal{E}}) \right),
    \end{equation}
    where $F_{v}$ denotes the completion at the place obtained restricting $v$ to $F$,
    \[
        C(j_{\mathcal{E}}) = \frac{1}{6} \left( \log 2 + \frac{22}{3} +   h\left(j_{\mathcal{E}}\right)\right),
    \]
    and \( j_{\mathcal{E}} \) is the \( j \)-invariant of \( \mathcal{E} \).
\end{lemma}

\begin{proof}
    We denote by \( M_{L} \) the set of all places of \( L \), by \( M_{L}^{\infty} \) the set of archimedean places, and by \( M_{L}^{0} \) the set of non-archimedean places.
    For each \( v \in M_{L} \), let
    \[
        d_{v} = \frac{[L_{v} : \mathbb{Q}_{v}]}{[L : \mathbb{Q}]}.
    \]
    By \cite[Theorem VI.2.1]{AdvancedTopicEllipticCurve}, we have
    \begin{equation*} \label{eq:height_decomposition}
        \hat{h}(aP - b\sigma(P)) = \sum_{v \in M_{L}} d_{v} \lambda_{v}(aP - b\sigma(P)).
    \end{equation*}
    Splitting the sum over different places, we obtain
    \begin{align} \label{eq:height_decomposition_split}
        \hat{h}(aP - b\sigma(P)) &= \sum_{v \in M_{L}^{\infty}} d_{v} \lambda_{v}(aP - b\sigma(P)) 
        + \sum_{\overset{v \in M_{L}^{0}}{v \nmid p  }} d_{v} \lambda_{v}(aP - b\sigma(P)) \notag \\
        &\quad + \sum_{v \in M_{L}^{0}, v \mid p} d_{v} \lambda_{v}(aP - b\sigma(P)).
    \end{align}
    
    Now the goal is to bound each of the three sums.
    
    We start with the archimedean contribution.  
    
    Since $aP-b\sigma(P) \neq 0$, by \cite[Corollary 2.3]{BakerMatthew}, we have
    \begin{align*}
        \sum_{v \in M_{L}^{\infty}} d_{v} \lambda_{v}(aP - b\sigma(P))
        &\ge -\sum_{v \in M_{L}^{\infty}} \frac {d_{w}}{6} \left( \log 2 + \frac{22}{3} + \log^{+}  |j_{\mathcal{E}}|_{w} \right)
    \end{align*}
   where $\log^{+}  |j_{\mathcal{E}}|_{w}=\log \left( \max \left(1,  |j_{\mathcal{E}}|_{w}\right) \right) $. Following \cite[Corollary~1.3.2]{bombieri2006heights} we have $\sum_{v \in M_{L}^{\infty}} d_{w}=1,$ and since \( j_{\mathcal{E}} \) is an algebraic integer  
(see \cite[Theorem~6.1]{AdvancedTopicEllipticCurve}), we also have $$\sum_{v \in M_{L}^{\infty}} d_{w}\log^{+}  |j_{\mathcal{E}}|_{w}=h(j_{\mathcal{E}}).$$ Thus, we have $$\sum_{v \in M_{L}^{\infty}} d_{v} \lambda_{v}(aP - b\sigma(P))\ge -\frac{1}{6} \left( \log 2 + \frac{22}{3}\right) -\frac {1}{6}  h(j_{\mathcal{E}})=-C(j_{\mathcal{E}}). $$

    Now we look at the non-archimedean contribution away from \( p \).
    By \cite[Remark VI.4.1.1]{AdvancedTopicEllipticCurve}, since \( \mathcal{E} \) has everywhere potential good reduction over \( F \), we have
    \begin{equation*} 
        \lambda_{v}(aP - b\sigma(P)) \geq 0, \quad \text{for all } v \in M_{L}^{0}.
    \end{equation*}

    Then we look at the places dividing $p$.
    Using assumption \eqref{eq:neron_local_lower_bound} and \cite[Lemma 2.1]{BakerMatthew}, we obtain
    \begin{equation*} \label{eq:place_p_lower_bound}
        \sum_{v \in M_{L}^{0}, v \mid p} d_{v} \lambda_{v}(aP - b\sigma(P))
        \geq \frac{\log p}{\rho} \sum_{v \in M_{L}^{0}, v \mid p} d_{v}
        = \frac{\log p \cdot [F_{\wp} : \mathbb{Q}_{p}]}{\rho [F : \mathbb{Q}]}.
    \end{equation*}
    
    By substituting the previous bounds into \eqref{eq:height_decomposition_split}, we obtain
    \begin{equation*} \label{eq:canonical_height_step1}
        \hat{h}(aP - b\sigma(P)) \geq \frac{\log p \cdot [F_{\wp} : \mathbb{Q}_{p}]}{\rho [F : \mathbb{Q}]} - C(j_{\mathcal{E}}).
    \end{equation*}
    
    By the properties of the canonical height, and the parallelogram law (see \cite[Theorem 9.3]{SilvermanArithmeticEC}), we have
    \begin{align*}
        \hat{h}(aP - b\sigma(P)) &= 2 \hat{h}(aP) + 2 \hat{h}(b\sigma(P)) - \hat{h}(aP + b\sigma(P)) \notag \\
        &= 2a^{2} \hat{h}(P) + 2b^{2} \hat{h}(P) - \hat{h}(aP + b\sigma(P)). \label{eq:parallelogram_law}
    \end{align*}
    Since \( \hat{h}(aP + b\sigma(P)) \geq 0 \), it follows that
    \begin{equation*} \label{eq:final_height_inequality}
        2(a^{2} + b^{2}) \hat{h}(P) \geq \frac{\log p \cdot [F_{\wp} : \mathbb{Q}_{p}]}{\rho [F : \mathbb{Q}]} - C(j_{\mathcal{E}}).
    \end{equation*}
    Dividing by \( 2(a^{2} + b^{2}) \) yields the desired inequality \eqref{eq:canonical_height_lower_bound}, completing the proof.
\end{proof}

The next step brings us close to Theorem~\ref{goal}, albeit under the additional assumption~\eqref{eq:tau_condition}.  
It can be viewed as the elliptic analogue of~\cite[Proposition 2.2]{Amoroso_David_Zannier-OnFieldWithPropertyB}.

\begin{prop} \label{Theorem intermediare}
    Let \( F \) be a number field and \( \mathcal{E} \) be an elliptic curve over $F$.  
    Suppose that \( \mathcal{E} \) has everywhere potential good reduction over \( F \), and let \( L/F \) be a finite abelian extension.  
    Let \( \wp \) be a prime ideal of \( F \) above a rational prime \( p \), and let \( q = N(\wp) \) be the norm of \( \wp \).  
    Let \( P \in \mathcal{E}(L) \setminus \mathcal{E}_{\mathrm{tors}} \), and suppose that for every \( \tau \in \mathrm{Gal}(L/F) \) of order \( p \), we have
    \begin{equation} \label{eq:tau_condition}
        \tau(P) - P \notin \mathcal{E}[p^{\infty}] \setminus \{0\}.
    \end{equation}
    Then, the canonical height satisfies
    \begin{equation*} \label{eq:height_lower_bound}
        \hat{h}(P) \geq \frac{1}{4q^{4}} \left( \frac{1}{d}\log q - C(j_{\mathcal{E}}) \right),
    \end{equation*}
    where \( d = [F:\mathbb{Q}] \) and
    \(
        C(j_{\mathcal{E}})
    \) is as in Lemma \ref{lemma minoration hauteur de neron},
    with \( j_{\mathcal{E}} \) being the \( j \)-invariant of \( \mathcal{E} \).
\end{prop}

\begin{proof}
    Let $\mathcal{O}_{L}$ be the ring of integers of $L$. Let \( \mathfrak{Q} \) be a prime ideal of \( \mathcal{O}_{L} \) over \( \wp \).  
    We define the decomposition group and ramification groups as follows:
    \begin{align*}
        G_{-1} &= G_{-1}(\mathfrak{Q} / \wp) = \{ \sigma \in \mathrm{Gal}(L/F) \mid \sigma(\mathfrak{Q}) = \mathfrak{Q} \}, \\
        G_k &= G_k(\mathfrak{Q} / \wp) = \{ \sigma \in \mathrm{Gal}(L/F) \mid \forall \gamma \in \mathcal{O}_L, \sigma \gamma \equiv \gamma \bmod \mathfrak{Q}^{k+1} \}.
    \end{align*}
    These form the filtration of subgroups:
    \[
        \mathrm{Gal}(L/F) \supseteq G_{-1} \supseteq G_0 \supseteq G_1 \supseteq \dots
    \]
    Moreover, for all \( k \geq 0 \), \( G_k \) is a normal subgroup of \( G_{-1} \) (see \cite[§4, Proposition 1, p. 62]{Hasse-ArfTheorem}).  
    Writing \( e := |G_0| = e(\mathfrak{Q}/\wp) = e_0 p^a \), where \( e(\mathfrak{Q}/\wp) \) is the ramification index of $\mathfrak{Q}$ over $\wp$ and \( \gcd(e_0, p) = 1 \), we have \( |G_0 / G_1| = e_0 \) (see \cite[§4, Corollary 1 an 3, p. 67]{Hasse-ArfTheorem}).

    Notice that from the proof of \cite[Proposition 2.3]{bogomolovPourKab} which rests on the Hasse-Arf Theorem \cite[§7, Theorem 1, p. 101]{Hasse-ArfTheorem}, the index \( k \) of the first trivial higher ramification group \( G_k \) satisfies
    \begin{equation} \label{eq:ramification_bound}
        k \geq \frac{e}{q}.
    \end{equation}
    
     \textbf{Case 1:} We suppose that \( \wp \) is unramified or tamely ramified in \( L \)  
    In this case, \( G_1 \) is trivial, so from \eqref{eq:ramification_bound}, we obtain \( e = e_0 \leq q \).  

    Suppose that there exists \( \tau \in G_0 \) such that \( qP \neq q\tau(P) \).    
   Since \( \tau(P) - P \in \mathcal{E}(\mathfrak{Q}) \), meaning that \( \tau(P)-P \) lies in the kernel of the reduction map
\(
E(\mathcal{O}_{L}) \longrightarrow E(\mathcal{O}_{L}/\mathfrak{Q}),
\)
in the sense of \cite[Chap.~VII.2]{SilvermanArithmeticEC}, and since \( e \leq q \), we obtain
\begin{equation*}
    q\left(\tau(P) - P\right) \in \mathcal{E}(\mathfrak{Q}^{q}) \subseteq \mathcal{E}(\mathfrak{Q}^{e}).
\end{equation*}

    Thus,
    \begin{equation*} \label{eq:neron_lower_bound}
        q(\tau(P) - P) \in \mathcal{E}( \wp\mathcal{O}_{L}).
    \end{equation*}
    Let \( v \in M_{F}^{0} \) be the place of $F$ associated to \( \wp \).  
    By \cite[Remark VI.4.1.1]{AdvancedTopicEllipticCurve}, we obtain
    \begin{equation*}
        \lambda_{v}(qP - q\tau(P)) \geq \frac{\log p}{e(\wp/p)}
    \end{equation*}
    where $e(\wp/p)$ is the ramification index of $\wp$ over $p$.
    
    Applying Lemma \ref{lemma minoration hauteur de neron} with \( a = b = q \) and \( \rho = e(\wp/p) \), we obtain
\begin{equation*} 
        \hat{h}(P) \geq \frac{1}{4q^{2}} \left(  \frac{ \log p \cdot [F_{v} : \mathbb{Q}_{v}]}{e(\wp/p)[F : \mathbb{Q}]} - C(j_{\mathcal{E}}) \right).
    \end{equation*}
Since $p^{ [F_{v} : \mathbb{Q}_{v}]/e(\wp/p)}=q$, and recalling that $d=[F:\Q]$, we get
    
    \begin{equation*} \label{eq:height_bound_case1}
       \hat{h}(P) \geq \frac{1}{4q^{2}} \left( \frac{1}{d}\log q - C(j_{\mathcal{E}}) \right).
    \end{equation*}

    Suppose now that for all \( \tau \in G_0 \), \( qP = \tau(qP) \). 
    Then, \( qP \in \mathcal{E}(L^{G_0}) \) and \( L^{G_0}/F \) is unramified at \( \wp \), where $L^{G_0}$ denote the subfield of $L$ fixed by the elements of $G_0$.  
    Let \( \phi \in G_{-1} \) be a lifting of the Frobenius automorphism. The extension $L/F$ is abelian so $\phi$ does not depend of the choice of the ideal $\mathfrak{Q}$ above $\wp$.  We have $$\forall \gamma \in \mathfrak{O}_{L^{G_{0}}}, \; \gamma^{q} \equiv \phi(\gamma) \mod \wp \mathfrak{O}_{L^{G_{0}}}.$$ Thus,
    \begin{equation*}
        q(qP) - \phi(qP) \in \mathcal{E}(\wp\mathcal{O}_{L^{G_0}}).
    \end{equation*}
    Since \( q(qP) \neq \phi(qP) \), otherwise we would have \( q^{4}\hat{h}(P)=\hat{h}(q(qP))=\hat{h}(\phi(qP))=q^{2}\hat{h}(P) \) which is a contradiction, we conclude that
    \begin{equation*}
        \lambda_{v}(q(qP) - \phi(qP)) \geq \frac{\log p}{e(\wp/p)}.
    \end{equation*}
    As \( \hat{h}(qP) = q^{2} \hat{h}(P) \), applying Lemma \ref{lemma minoration hauteur de neron} to the point $qP\in \mathcal{E}(L)$ with \( a = q \), \( b = 1 \), and \( \rho = e(\wp/p) \), we obtain
    \begin{equation*} 
        \hat{h}(P) \geq \frac{1}{2q^{2}(q^{2}+1)} \left( \frac{1}{d}\log q - C(j_{\mathcal{E}}) \right).
    \end{equation*}

     \textbf{Case 2:} Suppose that the prime ideal \( \wp \) is wildly ramified in \( L \).  
Let \( k \) be the index of the first trivial higher ramification group \( G_k \), so that \( G_{k-1} \) is an elementary \( p \)-group.  

Consider an element \( \tau \in G_{k-1} \) of order \( p \).  
If \( \tau(P) = P \), then \( P \in \mathcal{E}(L^{\tau}) \), where $L^{\tau}$ is the subfield of $L$ fixed by $\tau$, thus \( L^{\tau}/F \) is a proper abelian subextension of \( L/F \), allowing us to proceed by induction on the degree. 

Otherwise, since \( \tau(P) \neq P \) and since by \eqref{eq:tau_condition} \( \tau(P) - P \notin \mathcal{E}[p^{\infty}] \setminus \{0\} \), we conclude that
  \begin{equation*} 
      q(\tau(P) - P) \neq 0.
  \end{equation*}

By definition of \( \tau \) and by \eqref{eq:ramification_bound}, we also have
  \begin{equation*}
      q(\tau(P) - P) \in \mathcal{E}(\mathfrak{Q}^{kq}) \; \mbox{ with } \; kq\ge e .
  \end{equation*}
Applying Lemma \ref{lemma minoration hauteur de neron} with \( a = b = q \) and \( \rho = e(\wp/p) \), we get
    \begin{equation*} 
       \hat{h}(P) \geq \frac{1}{4q^{2}} \left( \frac{1}{d}\log q - C(j_{\mathcal{E}}) \right).
    \end{equation*}
\end{proof} 

One of the main challenges is to remove assumption~\eqref{eq:tau_condition}.  
A crucial step is provided by the following proposition.

\begin{prop} \label{question}
Let $\mathcal{E}$ be an elliptic curve with CM defined over a number field $F$, and let $q' = p^{\ell}$ be a power of a prime number $p$. Consider the natural Galois representation
\[
\rho: \mathrm{Gal}(F(\mathcal{E}[q'])/F) \longrightarrow \mathrm{Aut}(\mathcal{E}[q']) \simeq \mathrm{GL}_{2}(\Z/q'\Z).
\]
Then there exist $\sigma_1, \sigma_2 \in \mathrm{Gal}(F(\mathcal{E}[q'])/F)$ such that
\[
\rho(\sigma_i) = g_i I_2 \quad \text{for } i = 1, 2, \quad \text{with } 2 < g_2 - g_1 < 36[F : \Q] + 6,
\]
where $g_1,g_2 \in \Z$ and $I_2$ denotes the $2 \times 2$ identity matrix.

\end{prop}

\begin{proof}
   Let \( K \) be the complex multiplication field of the elliptic curve \( \mathcal{E} \), i.e. $\mathrm{End}(\mathcal{E})$ is an order in the imaginary quadratic field $K$. Let $\mathcal{O}_{K}$ be the ring of integers of $K$ ($\mathrm{End}(\mathcal{E}) \subset \mathcal{O}_{K}$), and define
\begin{align*}
    C(q') &:= \left(\mathcal{O}_K \otimes_{\mathbb{Z}} \mathbb{Z}/q'\mathbb{Z}\right)^{\times} \\
    &\simeq (\mathcal{O}_K / q'\mathcal{O}_K)^{\times} 
\end{align*}

viewed as a subgroup of \( \mathrm{GL}_{2}(\mathbb{Z}/q'\mathbb{Z}) \).

We consider the compositum $F'=FK$. It is then sufficient to show that the image of the Galois representation over \( F' \), i.e. $G := \rho\left(\mathrm{Gal}\left(F'(\mathcal{E}[q'])/F'\right)\right)$, contains two homotheties satisfying the desired condition.

Since \( K \subset F' \), by \cite[Corollaire on p.302]{JPserreGaloisImage}, we have that
\[
G \subset C(q').
\]
Since \( C(q') \) is abelian, \( G \) is a normal subgroup of \( C(q') \) and so we can consider the quotient group \( C(q')/G \).

Let $H=:\{ g \in \left( \Z/q'\Z \right)^{\times}, \; gId \in G \}$.
There is a natural injective group homomorphism
\[
\left\{ \lambda \cdot \mathrm{Id} \mid \lambda \in (\mathbb{Z}/q'\mathbb{Z})^{\times} \right\}/{ \left\{ g \cdot \mathrm{Id} \mid g \in H \right\} } \hookrightarrow C(q')/G,
\]
obtained by composing the obvious inclusion
\[
\left\{ \lambda \cdot \mathrm{Id} \mid \lambda \in (\mathbb{Z}/q'\mathbb{Z})^{\times} \right\} \hookrightarrow C(q')
\]
with the projection \( C(q') \twoheadrightarrow C(q')/G \), whose kernel is by definition \[ \{ g \cdot \mathrm{Id} \mid g \in H \}. \]

It follows that
\[
[(\mathbb{Z}/q'\mathbb{Z})^{\times} : H] = \left| \frac{\left\{ \lambda \cdot \mathrm{Id} \mid \lambda \in (\mathbb{Z}/q'\mathbb{Z})^{\times} \right\}}{ \left\{ g \cdot \mathrm{Id} \mid g \in H \right\} } \right| \le [C(q') : G] \le 3 [F' : \mathbb{Q}],
\]
where the last inequality is a direct consequence of \cite[Theorem 1.5 point 1.]{TheseDavide}, which makes completely explicit the result of Serre \cite[Corollaire on p.302]{JPserreGaloisImage}.

Applying \cite[Lemma 4.1]{Amoroso_David_Zannier-OnFieldWithPropertyB}, and noticing that $[F':\Q]\le 2[F:\Q]$, yields the existence of elements \( g_1, g_2 \in H \) such that
\[
2 < g_2 - g_1 < 36 [F : \mathbb{Q}] + 6,
\]
thus concluding the proof.

\end{proof}

\section{Proof of Theorem \ref{goal}} \label{proof of goal}
We now have all the tools needed to prove our main result.

\begin{proof}[Proof of Theorem \ref{goal}]
  Let $p$ be a prime number which will be fixed later. We retain the notations from Theorem \ref{goal} and let $q' = p^{\ell}$ such that
  \[
  L \cap \Q(\mathcal{E}[p^{\infty}]) \subset L \cap \Q(\mathcal{E}[q']),
  \]
  where $L$ is a finite abelian extension of $F$, and $\mathcal{E}[p^{\infty}]$ denotes the group of torsion points of~$\mathcal{E}$ 
of order a power of~$p$.
By Proposition \ref{question} there exist $\sigma_1, \sigma_2 \in \mathrm{Gal}(F(\mathcal{E}[q'])/F)$ such that
  \[
  \rho(\sigma_i) = g_i I_2 \quad \text{for } i = 1, 2, \quad \text{with } 2 < g_2 - g_1 < 36d+6.
  \]
  Let $\tilde{\sigma}_i \in \mathrm{Gal}(L(\mathcal{E}[q'])/F)$ be extensions of $\sigma_i$, and let $P \in \mathcal{E}(L) \setminus \mathcal{E}_{\mathrm{tors}}$. Define
  \[
  Q := \tilde{\sigma}_2(P) - \tilde{\sigma}_1(P) - [\,g_2 - g_1\,] P.
  \]
 We first want to show that $Q$ satisfies the hypothesis of Proposition \ref{Theorem intermediare}.

  Suppose, by contradiction, that $Q \in \mathcal{E}_{\mathrm{tors}}$. Then, by the usual properties of the Néron–Tate height,
  \[
  \hat{h}(\tilde{\sigma}_1(P) - \tilde{\sigma}_2(P)) = (g_2 - g_1)^2 \hat{h}(P).
  \]
  However, the parallelogram law also gives
  \[
  \hat{h}(\tilde{\sigma}_1(P) - \tilde{\sigma}_2(P)) + \hat{h}(\tilde{\sigma}_1(P) + \tilde{\sigma}_2(P)) = 2\hat{h}( \sigma_1(P))+2\hat{h}( \sigma_2(P)) = 4 \hat{h}(P),
  \]
  implying
  \[
  \hat{h}(\tilde{\sigma}_1(P) + \tilde{\sigma}_2(P)) = (4 - (g_2 - g_1)^2) \hat{h}(P).
  \]
  Since $2 < g_2 - g_1$ and $P \notin \mathcal{E}_{\mathrm{tors}}$, the right-hand side is negative, which gives a contradiction. Therefore, $Q \notin \mathcal{E}_{\mathrm{tors}}$.

  Suppose now, by contradiction, that there exists $\tau \in \mathrm{Gal}(L/F)$ such that
  \[
  \tau(Q) - Q \in \mathcal{E}[p^{\infty}] \setminus \{0\}.
  \]
  Set $T := \tau(Q) - Q$ and $R := \tau(P) - P$. Let $K$ be the complex multiplication field of the elliptic curve $\mathcal{E}$. Let's assume for the moment that $F$ contains the Hilbert class field of $K$. This ensures by \cite[Theorem II.2.3]{AdvancedTopicEllipticCurve} that the group $\mathrm{Gal}(L(E[q'])/F)$ is abelian.

Therefore, we have
  \begin{align*}
  T &= \tau \left(\tilde{\sigma}_2(P) - \tilde{\sigma}_1(P) - [\,g_2 - g_1\,] P \right)-\tilde{\sigma}_2(P) - \tilde{\sigma}_1(P) - [\,g_2 - g_1\,] P \\
  &=\tilde{\sigma}_2(\tau(P) - P) - \tilde{\sigma}_1(\tau(P) - P) - [\,g_2 - g_1\,](\tau(P) - P) \\
  &=\tilde{\sigma}_2(R) - \tilde{\sigma}_1(R) - [\,g_2 - g_1\,] R.
  \end{align*}
  Since $T \in \mathcal{E}_{\mathrm{tors}}$, the same arguments as above yields that we must have $\hat h(R)=0$ and so $R \in \mathcal{E}_{\mathrm{tors}}$.
  Write $R = R_1 + R_2$, where $R_1$ is a torsion point of order prime to $p$, and $R_2 \in \mathcal{E}[p^{\infty}]$ of order $p^k$ for some $k \in \Z$. By Bézout's identity, there exist $u, v \in \Z$ such that $mv + p^k u = 1$, so
  \[
  R_2 = (1 - p^k u) R_2 = mv R \in \mathcal{E}(\Q(R)) \subset \mathcal{E}(L \cap \Q(\mathcal{E}[p^{\infty}])) \subset \mathcal{E}(F(\mathcal{E}[q'])).
  \]
  Since, by the choice of $\sigma_1$ and $\sigma_2$, we have 
$\sigma_1(S) = g_1 S$ and $\sigma_2(S) = g_2 S$ for every 
$S \in \mathcal{E}(F(\mathcal{E}[q']))$, it follows that
\[
  \tilde{\sigma}_2(R_2) \;-\; \tilde{\sigma}_1(R_2) \;-\; [\,g_2 - g_1\,] R_2 
  \;=\; 0,
\]

  and so
  \[
  T = \tilde{\sigma}_2(R_1) - \tilde{\sigma}_1(R_1) - [\,g_2 - g_1\,] R_1.
  \]
  Thus, $T$ is a torsion point of order prime to $p$, contradicting that $T \in \mathcal{E}[p^{\infty}] \setminus \{0\}$. Hence, $Q$ satisfies the hypotheses of Proposition \ref{Theorem intermediare}. 
  
  Let $\wp$ be a prime ideal of $F$ lying above $p$ and $q=N(\wp)$.  Noticing that $\mathcal{E}$ has everywhere potential good reduction, due to the complex multiplication \cite[Theorem II.6.4]{AdvancedTopicEllipticCurve}, we have by Proposition \ref{Theorem intermediare} that
  \begin{align*}
  \hat{h}(Q) \geq \frac{1}{4 q^4} \left(\frac{1}{d} \log q -C(j_{\mathcal{E}})\right),
  \end{align*}
  where $d=[F:\Q]$.

  We also have by the properties of the canonical height ( \cite[Theorem 9.3]{SilvermanArithmeticEC} ),
  \begin{align*}
  \hat{h}(Q)&= \hat{h}(\tilde{\sigma}_2(P) - \tilde{\sigma}_1(P) - [\,g_2 - g_1\,] P)  \\& \leq 2(\hat{h}(\tilde{\sigma}_2(P) - \tilde{\sigma}_1(P))+\hat{h}( [\,g_2 - g_1\,] P))\\
  &\leq 2(4 + (g_2 - g_1)^2) \hat{h}(P),
  \end{align*}
  hence, recalling that $2<g_{2}-g_{1}<36d+6$,
  \begin{align*}
  \hat{h}(P) &\geq \left(8(4 + (g_2 - g_1)^2) q^4\right)^{-1}  \left(\frac{1}{d} \log q -C(j_{\mathcal{E}})\right)\\
  &\geq \left((5+27d+162d^2)4^{3}q^4 \right)^{-1} \left(\frac{1}{d} \log q -C(j_{\mathcal{E}})\right).
  \end{align*}

  Now, choose a prime number $p$ such that
  \[
  2e^{dC(j_{\mathcal{E}})} \leq p < 4e^{dC(j_{\mathcal{E}})}.
  \]
   \noindent Then
$$
2e^{dC(j_E)}\leq p\leq q\leq p^d<4^de^{d^2C(j_E)}
$$
and thus 
$$
\tfrac1d\log q - C(j_E)\geq \tfrac1d\log 2\geq \tfrac1{4d}.
$$
Therefore, taking into account $(5+27d+162d^2)d \leq 2^{8d}$ and $e^4\leq 2^6$,
\begin{align*}
\hat{h}(P)
&\geq\big((5+27d^2+162d^3)4^3q^4\big)^{-1}\big(\tfrac1d\log q - C(j_E)\big)\\
&\geq\big((5+27d^2+162d^3)d4^{4d+4}e^{4d^2C(j_E)}\big)^{-1}\\
&\geq 2^{-16d-8-6d^2C(j_E)}.
\end{align*}
By definition of $C(j_E)$ we have 
$$
2^{6C(j_E)}=2^{\log 2+ 22/3+h(j_E)}\leq 2^{9+h(j_E)}.
$$
Thus we finally get:
$$
\hat{h}(P)\geq  2^{-(9+h(j_E))d^2-16d-8}.
$$
 Now, if the field $F$ does not contain the Hilbert class field of $K$, 
we may replace $F$ by its compositum with $K$, i.e. $FK$. 
Indeed, since the elliptic curve $\mathcal{E}$ is defined over $F$, 
the field $F$ contains its $j$-invariant. 
Consequently, $FK$ contains $K(j_{\mathcal{E}})$, 
which, by \cite[Theorem~II.4.1]{AdvancedTopicEllipticCurve}, is the Hilbert class field of $K$. 
Thus, the preceding argument can be carried out over this field.
Since $K$ is quadratic, it follows that
\begin{align*}
    [F K : \mathbb{Q}]
        &\leq [F : \mathbb{Q}][K : \mathbb{Q}] \\
        &= 2d.
\end{align*}

Hence,
\begin{align*}
  \hat{h}(P) &\ge 2^{-(9+h(j_{\mathcal{E}}))(2d)^{2}-32d-8} \\
  &= 2^{-4(9+h(j_{\mathcal{E}}))d^{2}-32d-8},
\end{align*}
which completes the proof.

\end{proof}

\section*{Acknowledgment}

The author would like to thank Samuel Le Fourn and Davide Lombardo for providing the proof of Proposition \ref{question} and for valuable discussions, and also Gabriel Dill for pointing out several inaccuracies.

\bibliographystyle{alpha}
\bibliography{bibliography}

\end{document}